\documentclass[a4paper,12pt]{article}

\usepackage[utf8]{inputenc} 
\usepackage[english]{babel}
\usepackage{amsfonts,amsmath,amsthm,amssymb}
\usepackage{a4}
\usepackage[all]{xy}

\title{\sc{Kazhdan's property $(T)$ with respect to non-commutative $L_{p}$-spaces}}

\theoremstyle{proclaim}
\newtheorem{thm}{Theorem}[section]
\newtheorem{df}[thm]{Definition}
\newtheorem{prp}[thm]{Proposition}
\newtheorem{lem}[thm]{Lemma}
\newtheorem{rem}[thm]{Remark}
\newtheorem{cor}[thm]{Corollary}

\newcommand{\esp}{\textrm{ }}
\newcommand{\n}{\vert\vert}

\author{\sc{Baptiste OLIVIER}}

\begin{document}
\maketitle

\begin{abstract}
We show that a group with Kazhdan's property $(T)$ has property $(T_{B})$ for $B$ the Haagerup non-commutative $L_{p}(\mathcal{M})$-space associated with a von Neumann algebra $\mathcal{M}$, $1<p<\infty$. We deduce that higher rank groups have property $F_{L_{p}(\mathcal{M})}$.
\end{abstract}


\section{\sc{Introduction}}

Kazhdan's property $(T)$ of a topological group $G$ is an important rigidity property, defined in terms of the unitary representations of $G$ on Hilbert spaces. We recall the precise definition :
\begin{df}
\rm{}A pair $(G,H)$ of topological groups, where $H$ is a closed subgroup of $G$, is said to have relative property $(T)$ if there exist a compact subset $Q$ of $G$ and $\epsilon>0$ such that : whenever a unitary representation $\pi$ of $G$ on a Hilbert space $\mathcal{H}$ has a $(Q,\epsilon)$-invariant vector, that is a vector $\xi\in\mathcal{H}$ such that $$\sup_{g\in Q}\vert\vert\pi(g)\xi-\xi\vert\vert<\epsilon\vert\vert\xi\vert\vert$$ then $\pi$ has a non-zero $\pi(H)$-invariant vector.   The pair $(Q,\epsilon)$ is called a Kazhdan pair.\\
A topological group $G$ is said to have property $(T)$ if the pair $(G,G)$ has relative property $(T)$. 
\end{df}
For more details on property $(T)$, see the monography \cite{bekka2008kazhdan}.\\
The following variant of this property for Banach spaces was recently introduced by Bader, Furman, Gelander and Monod in \cite{bader2007propertyTLp}. Let $B$ be a Banach space and $O(B)$ the orthogonal group of $B$, that is, the group of linear bijective isometries of $B$. Recall that an orthogonal representation of a topological group $G$ on a Banach space $B$ is a homomorphism $\rho:G\rightarrow O(B)$ such that the map $g\mapsto \rho(g)x$ is continuous for every $x\in B$. If $\rho:G\rightarrow O(B)$ is an orthogonal representation of a group $G$, we denote the subspace of $\rho(G)$-invariant vectors by $$B^{\rho(G)}=\{x\in B\esp\vert \esp\rho(g)x=x\textrm{ for all }g\in G\esp\}.$$ Observe that $B^{\rho(G)}$ is invariant under $G$. The representation $\rho$ is said to almost have invariant vectors if it has $(Q,\epsilon)$-invariant vector for every compact subset $Q$ of $G$ and $\epsilon>0$. 
\begin{df}
\rm{}Let $G$ be a topological group and $H$ be a closed $normal$ subgroup of $G$. The pair $(G,H)$ has relative property $(T_{B})$ for a Banach space $B$ if, for any orthogonal representation $\rho:G\rightarrow O(B)$, the quotient representation $\rho{'}:G\rightarrow O(B/B^{\rho(H)})$ does not almost have $\rho{'}(G)$-invariant vectors.\\
A topological group $G$ has property $(T_{B})$ if the pair $(G,G)$ has relative property $(T_{B})$.
\end{df}
The authors of \cite{bader2007propertyTLp} studied the case where $B$ is a superreflexive Banach space, and among other things, they showed that a group which has property $(T)$ has  property $(T_{L^{p}(\mu)})$ for $\mu$ a $\sigma$-finite measure on a standard Borel space $(X,\mathcal{B})$ and $1<p<\infty$. We will extend this result to the non- commutative setting.\\
Non-commutative $L_{p}$-spaces were introduced by Dixmier \cite{Dixmier1953introductionLp} and studied by various authors, among them Yeadon \cite{Yeadon1975Lpspaces} and Haagerup \cite{Haagerup1979Lpspaces} (for a survey on these spaces, see Pisier and Xu \cite{Pisier2003Lpspaces}). Apart from the standard $L^{p}(\mu)$-spaces, common examples are the  $p$-Schatten ideals
$$S_{p}=\{x\in\mathcal{B}(\mathcal{H})\esp\vert\esp {\rm tr}(\vert x\vert^{p})<\infty\esp\}$$ where $\mathcal{H}$ is a separable Hilbert space.\\
We review below (in Section 2) Haagerup's definition of these non-commutative $L_{p}$-spaces. Here is our main result :
\begin{thm}\label{thm1}
Let $G$ be a topological group and $H$ a closed normal subgroup of $G$. Assume that the pair $(G,H)$ has relative property $(T)$. For every von Neumann algebra $\mathcal{M}$, the pair $(G,H)$ has relative property $(T_{L_{p}(\mathcal{M})})$ for $1<p<\infty$.
\end{thm}
In particular, if $G$ has property $(T)$, then $G$ has property $(T_{L_{p}(\mathcal{M})})$ for \\$1<p<\infty$. Property $(T_{B})$ has a stronger version which is a fixed point property for affine actions.
\begin{df}
\rm{}Let $B$ a Banach space. A topological group $G$ has property $(F_{B})$ if every continuous action of $G$ by affine isometries on $B$ has a $G$-fixed point.
\end{df}
The authors of \cite{bader2007propertyTLp} showed that higher rank groups and their lattices have property $(F_{L^{p}(\mu)})$. 
\begin{df}
\rm{}For $1\leq i\leq m$, let $k_{i}$ be local fields and $\mathbb{G}_{i}(k_{i})$ be the $k_{i}$-points of connected simple $k_{i}$-algebraic groups $\mathbb{G}_{i}$. Assume that each simple factor $\mathbb{G}_{i}$ has $k_{i}$-rank $\geq2$. The group $G=\Pi_{i=1}^{m}\mathbb{G}_{i}(k_{i})$ is called a higher rank group.
\end{df}
Our next result shows that Theorem B in \cite{bader2007propertyTLp} remains true for non-commutative $L_{p}$-spaces. 
\begin{thm}\label{thm2}
Let $G$ be a higher rank group and $\mathcal{M}$ a von Neumann algebra. Then $G$, as well as every lattice in $G$, has property $F_{L_{p}(\mathcal{M})}$ for $1<p<\infty$.
\end{thm}
Theorem 1.6 was proved by Puschnigg in \cite{puschnigg2008finitely} in the case $L_{p}(\mathcal{M})=S_{p}$. The strategy of the proof of Theorem \ref{thm1} (as in \cite{puschnigg2008finitely}) follows the one from \cite{bader2007propertyTLp}. To achieve the result, we will need some results on the Mazur map and the description of the surjective isometries of $L_{p}(\mathcal{M})$ given by Sherman in \cite{sherman2005isometries}.\\
The paper is organized as follows. In Section 2, useful properties of the Mazur map are established. Group representations on $L_{p}(\mathcal{M})$ are studied in Section 3. The proof of Theorem \ref{thm1} is given in Section 4. In Section 5, we show how Theorem \ref{thm2} can be obtained from a variant of  Theorem \ref{thm1}.

\section{\sc{Some properties of the Mazur map}}Let $\mathcal{M}$ be a von Neumann algebra, acting on a Hilbert space $\mathcal{H}$, and equipped with a normal semi-finite weight $\varphi_{0}$. Let $t\mapsto\sigma^{\varphi_{0}}_{t}$ be the one-parameter group of modular automorphisms of $\mathcal{M}$ with respect to $\varphi_{0}$. We denote by $\mathcal{N}_{\varphi_{0}}=\mathcal{M}\rtimes_{\varphi_{0}}\mathbb{R}$ the crossed product von Neumann algebra, which is a von Neumann algebra acting on $L^{2}(\mathbb{R},\mathcal{H})$, and  generated by the operators $\pi_{\varphi_{0}}(x)$, $x\in\mathcal{M}$, and $\lambda_{s}$, $s\in\mathbb{R}$, defined by
\begin{displaymath}
\begin{split}
&\pi_{\varphi_{0}}(x)(\xi)(t)=\sigma^{\varphi_{0}}_{-t}(x)\xi(t)\\
&\lambda_{s}(\xi)(t)=\xi(t-s)\esp\esp\esp\esp\esp\esp\esp\textrm{for any }\xi\in L^{2}(\mathbb{R},\mathcal{H})\textrm{ and }t\in\mathbb{R}.
\end{split}
\end{displaymath} 
There is a dual action $s\mapsto\theta_{s}$ of $\mathbb{R}$ on $\mathcal{N}_{\varphi_{0}}$. Then let $\tau_{\varphi_{0}}$ be the semi-finite normal trace on $\mathcal{N}_{\varphi_{0}}$ satisfying
$$\tau_{\varphi_{0}}\circ\theta_{s}={\rm e}^{-s}\tau_{\varphi_{0}}\textrm{for all }s\in\mathbb{R}.$$
We denote by $\ L_{0}(\mathcal{N}_{\varphi_{0}},\tau_{\varphi_{0}})$ the *-algebra of $\tau_{\varphi_{0}}$-measurable operators affiliated with $\mathcal{N}_{\varphi_{0}}$. For $1\leq p\leq\infty$, the Haagerup non-commutative $L_{p}$-space associated with $\mathcal{M}$ is defined by
$$L_{p}(\mathcal{M})=\{\esp x\in L_{0}(\mathcal{N}_{\varphi_{0}},\tau_{\varphi_{0}})\esp\vert\esp\theta_{s}(x)={\rm e}^{-s/p}x\textrm{ for all }s\in\mathbb{R}\}.$$
It is known that this space is independant of the weight $\varphi_{0}$ up to isomorphism. The space $L_{1}(\mathcal{M})$ is isomorphic to $\mathcal{M}_{*}$. The identification goes as follows : there exists a normal faithful semi-finite operator valued weight from $\mathcal{N}_{\varphi_{0}}$ to $\mathcal{M}$ defined by
$$\Phi_{\varphi_{0}}(x)=\pi_{\varphi_{0}}^{-1}(\int_{\mathbb{R}}\theta_{s}(x){\rm d}s)\esp,\textrm{ for }x\in\mathcal{N}_{\varphi_{0}}.$$
Now, if $\varphi\in\mathcal{M}_{*}^{+}$, and $\hat{\varphi}$ denotes the extension of $\varphi$ to a a normal weight on $\Hat{\mathcal{M}}^{+}$, the extended positive part of $\mathcal{M}$, we then put
$$\tilde{\varphi}^{\varphi_{0}}=\hat{\varphi}\circ\Phi_{\varphi_{0}}.$$
We associate to $\varphi$ the Radon-Nikodym derivative $\frac{d\tilde{\varphi}^{\varphi_{0}}}{d\tau_{\varphi_{0}}}$ of $\tilde{\varphi}^{\varphi_{0}}$ with respect to the trace $\tau_{\varphi_{0}}$. This isomorphism between $\mathcal{M}_{*}^{+}$ and $L_{1}(\mathcal{M})^{+}$ extends to the whole spaces by linearity.\\
If $x\in L_{1}(\mathcal{M})$, and $\varphi_{x}$ is the element of  $\mathcal{M}_{*}^{+}$ associated to $x$, we define a linear functional Tr by
$${\rm Tr}(x)=\varphi_{x}(1)$$
and we have, $p'$ being the conjugate exponent of $p$,
$${\rm Tr}(xy)={\rm Tr}(yx)\textrm{ for }x\in L_{p}(\mathcal{M}),\esp y\in L_{p'}(\mathcal{M})$$
For $1\leq p<\infty$, if $x=u\vert x\vert$ is the polar decomposition of $x\in L_{p}(\mathcal{M})$, we define 
$$\vert\vert x\vert\vert_{p}={\rm Tr}(\vert x\vert^{p})^{1/p}.$$
Equipped with $\vert\vert . \vert\vert_{p}$, $L_{p}(\mathcal{M})$ is a Banach space. For $1<p<\infty$, the dual space of $L_{p}(\mathcal{M})$ is $L_{p'}(\mathcal{M})$ and $L_{p}(\mathcal{M},\tau)$ is known to be superreflexive.\\

We now introduce the Mazur map and establish some of its properties. 
\begin{df}
\rm{}Let $1\leq p,q<\infty$. For an operator $a$, let $\alpha\vert a\vert$ be its polar decomposition. The map
\begin{displaymath}
\begin{split}
M_{p,q}:&L_{0}(\mathcal{N}_{\varphi_{0}},\tau_{\varphi_{0}})\rightarrow L_{0}(\mathcal{N}_{\varphi_{0}},\tau_{\varphi_{0}})\\
               &x=\alpha\vert a\vert\mapsto\alpha\vert a\vert^{\frac{p}{q}}
\end{split}
\end{displaymath}
is called the Mazur map.
\end{df}

We will need the following lemma.
\begin{lem}\label{lem3.4}
Let $1\leq p,q,r<\infty$. Then $M_{r,q}\circ M_{p,r}=M_{p,q}$.
\end{lem}
\begin{proof}
Let $\alpha\vert x\vert$ be the polar decomposition of $x\in L_{0}(\mathcal{N}_{\varphi_{0}},\tau_{\varphi_{0}})$. Let $\beta>0$, and set $y=\alpha\vert x\vert^{\beta}$. We claim that the polar decomposition of $y$ is given by $\alpha$ and $\vert x\vert^{\beta}$. To show this, it suffices to prove that $\overline{ {\rm Im}(\vert x\vert^{\beta})}=\overline{ {\rm Im}(\vert x\vert)}$.\\
By taking orthogonals, we have to show that ${\rm Ker}(\vert x\vert)={\rm Ker}(\vert x\vert^{\beta})$ for all $\beta>0$. Recall that the domain $D(\vert x\vert^{\beta})$ of $\vert x\vert^{\beta}$ is
$$D(\vert x\vert^{\beta})=\{\xi\esp\vert\esp\int_{0}^{\infty}\lambda^{2\beta}d\mu_{\xi}(\lambda)<\infty\}.$$
If $\xi\in{\rm Ker}  (\vert x\vert)$, we have for all $\eta\in L^{2}(\mathbb{R},\mathcal{H})$
$$<\vert x\vert\xi,\eta>=\int_{0}^{\infty}\lambda d\mu_{\xi,\eta}(\lambda)=0.$$
In particular, $\mu_{\xi}(]0,\infty[)=0$. So $\xi\in D(\vert x\vert^{\beta})$ and $\xi\in {\rm Ker}(\vert x\vert^{\beta})$ thanks to
$$<\vert x\vert^{\beta}\xi,\eta>=\int_{0}^{\infty}\lambda^{\beta} d\mu_{\xi,\eta}(\lambda)=0.$$
By exchanging the role of $\vert x\vert$ and $\vert x\vert^{\beta}$, we get the equality.  \\
Let $1\leq p,q,r<\infty$, and $\beta=p/r$; then $M_{p,r}(x)=\alpha\vert x\vert^{\beta}$. It follows from what we have just seen that $M_{r,q}(M_{p,r}(x))=\alpha\vert x\vert^{\frac{p}{q}}=M_{p,q}(x)$.
\end{proof}

\begin{prp}\label{prp3.5}
Let $1\leq p,q<\infty$, and $a\in L_{p}(\mathcal{M})$. Then
\begin{displaymath}
\vert\vert M_{p,q}(a)\vert\vert_{q}^{q}=\vert\vert a\vert\vert_{p}^{p}.
\end{displaymath}
\end{prp}
\begin{proof}
We denote again by $\alpha \vert a\vert$ the polar decomposition of $a$. We have already seen that $\vert M_{p,q}(a)\vert=\vert a\vert^{\frac{p}{q}}$. So we have 
\begin{displaymath}
{\rm Tr}(\vert M_{p,q}(a)\vert^{q})={\rm Tr}(\vert a\vert^{p}).
\end{displaymath} 
\end{proof}

\begin{prp}\label{prp3}
Let $p,q\in]1,\infty[$ be conjugate. The map
\begin{displaymath}
\begin{split}
L_{p}(\mathcal{M})&\rightarrow L_{q}(\mathcal{M})\\
x&\mapsto M_{p,q}(x)^{*}
\end{split}
\end{displaymath}
is the duality map from $L_{p}(\mathcal{M})$ to $L_{q}(\mathcal{M})$. 
\end{prp}
\begin{proof}
We first notice that $M_{p,q}$ sends $L_{p}(\mathcal{M})$ into $L_{q}(\mathcal{M})$. Let $x=\alpha\vert x\vert\in L_{p}(\mathcal{M})$ and $s\in\mathbb{R}$. By uniqueness in the polar decomposition, we have $\theta_{s}(\alpha)=\alpha$ and $\theta_{s}(\vert x\vert)={\rm e}^{-s/p}\vert x\vert$, and then
\begin{displaymath}
\begin{split}
\theta_{s}(M_{p,q}(x))&=\theta_{s}(\alpha)\theta_{s}(\vert x\vert^{\frac{p}{q}})\\
&=\alpha(\theta_{s}(\vert x\vert)^{\frac{p}{q}})\\
&={\rm e}^{-s/q}M_{p,q}(x).
\end{split}
\end{displaymath}
Thanks to the uniqueness of the duality map in superreflexive spaces, we just have to check that ${\rm Tr}(M_{p,q}(a)^{*}a)=1$ for $a$ in the unit sphere $S(L_{p}(\mathcal{M}))$ of $L_{p}(\mathcal{M})$.\\
Let $a=\alpha\vert a\vert\in S(L_{p}(\mathcal{M}))$; then $M_{p,q}(a)=\alpha\vert a\vert^{\frac{p}{q}}$. Since $\alpha^{*}\alpha \vert a\vert=\vert a\vert$, it follows that 
$${\rm Tr}(\vert a\vert^{\frac{p}{q}}\alpha^{*}\alpha\vert a\vert)={\rm Tr}(\vert a\vert^{\frac{p}{q}}\vert a\vert)={\rm Tr}(\vert a\vert^{p})=1.$$
\end{proof}

\begin{prp}\label{prp3.7}
If $a,b\in L_{0}(\mathcal{N}_{\varphi_{0}},\tau_{\varphi_{0}})$ and if $e,f$ are two central projections in $\mathcal{N}_{\varphi_{0}}$ such that $ef=0$, then $M_{p,q}(ae+bf)=M_{p,q}(ae)+M_{p,q}(bf)$.
\end{prp}
\begin{proof}
As is easily checked, we have $$\vert ae+bf\vert=\vert a\vert e+\vert b\vert f.$$
Let $\gamma$ be the partial isometry occuring in the polar decomposition of $ae+bf$, and let $a=\alpha\vert a\vert$, $b=\beta\vert b\vert$ be the polar decompositions of $a$ and $b$. We claim that $\gamma=\alpha e+\beta f$. Indeed, we have
\begin{displaymath}
\begin{split}
&ae+bf=\gamma\vert ae+bf\vert \\
&\textrm{and }ae+bf=(\alpha e)(\vert a\vert e)+(\beta f)(\vert b\vert f)=(\alpha e+\beta f)\vert ae+bf\vert.
\end{split}
\end{displaymath}
Since $\alpha e$ is zero on ${\rm Ker}(\vert a\vert e)$ and $\beta f$ is zero on ${\rm Ker}(\vert b\vert f)$, $\alpha e+\beta f$ is zero on ${\rm Im}(\vert ae +bf\vert)^{\bot}={\rm Ker}(\vert ae+by\vert)={\rm Ker}(\vert a\vert e)\cap {\rm Ker}(\vert b\vert f)$ ($ef=0$).\\
Using again the fact that $ef=0$ and that $e,f$ are central elements, we deduce that
\begin{displaymath}
\begin{split}
M_{p,q}(ae+bf)&=(\alpha e+\beta f)\vert ae+bf\vert^{\frac{p}{q}}\\
&=(\alpha e+\beta f)(e\vert a\vert^{\frac{p}{q}}+f\vert b\vert^{\frac{p}{q}})\\
&=M_{p,q}(ae)+M_{p,q}(bf).
\end{split}
\end{displaymath}
\end{proof}

\begin{prp}\label{maz}
Let $J$ be a Jordan-isomorphism of $\mathcal{N}_{\varphi_{0}}$, and let $1\leq p,q<\infty$. Then we have
$$J(x)=M_{p,q}\circ J\circ M_{q,p}(x)\textrm{ for all }x\in\mathcal{N}_{\varphi_{0}}.$$
\end{prp}
\begin{proof}
By Lemma 3.2 in \cite{stormer1965jordanmorphisms}, we have a decomposition $J=J_{1}+J_{2}$ with the following properties : $J_{1}$ is a {*}-homomorphism, $J_{2}$ is a {*}-anti-homomorphism and $J_{1}(x)=J(x)e$, $J_{2}(x)=J(x)f$ for all $x\in\mathcal{M}$, with $e,f$ two orthogonal and central projections such that $e+f=I$.\\
Observe first that, for $a\in\mathcal{N}_{\varphi_{0}}$ with $a\geq0$ and a positive real number $r$, we have $$J_{1}(a^{r})=J_{1}(a)^{r}$$
and the same is true for $J_{2}$.\\
If $\alpha$ is a partial isometry, then $J_{1}(\alpha)$ and $J_{2}(\alpha)$ are partial isometries with initial supports $J_{1}(\alpha^{*}\alpha)$ and $J_{2}(\alpha\alpha^{*})$, and final supports $J_{1}(\alpha\alpha^{*})$) and $J_{2}(\alpha^{*}\alpha)$) respectively.\\
Let $x=\alpha\vert x\vert\in\mathcal{N}_{\varphi_{0}} $. Since the supports of $J_{1}$ and $J_{2}$ are orthogonal, it follows from Proposition \ref{prp3.7} that
\begin{displaymath}
\begin{split}
M_{p,q}\circ J\circ M_{q,p}(x)&=M_{p,q}(J_{1}(M_{q,p}(x))+J_{q}(M_{q,p}(x)))\\
&=M_{p,q}(J_{1}(M_{q,p}(x)))+M_{p,q}(J_{2}(M_{q,p}(x))).
\end{split}
\end{displaymath}
Moreover, we have
\begin{displaymath}
\begin{split}
M_{p,q}(J_{1}(M_{q,p}(x)))&=M_{p,q}(J_{1}(\alpha\vert x\vert^{\frac{2}{p}}))\\
&=M_{p,q}(J_{1}(\alpha)J_{1}(\vert x\vert)^{\frac{2}{p}})\\
&=J_{1}(x)
\end{split}
\end{displaymath}
and
\begin{displaymath}
\begin{split}
M_{p,q}(J_{2}(M_{q,p}(x)))&=M_{p,q}(J_{2}(\alpha\vert x\vert^{\frac{2}{p}}\alpha^{*}\alpha))\\
&=M_{p,q}(J_{2}(\alpha)J_{2}(\alpha\vert x\vert^{\frac{2}{p}}\alpha^{*}))\\
&=M_{p,q}(J_{2}(\alpha)J_{2}((\alpha\vert x\vert\alpha^{*})^{\frac{2}{p}}))\\
&=M_{p,q}(J_{2}(\alpha)J_{2}(\alpha\vert x\vert\alpha^{*})^{\frac{2}{p}})\\
&=J_{2}(x).
\end{split}
\end{displaymath}
\end{proof}

An essential tool for the proof of Theorem \ref{thm1} is the following result about the local uniform continuity of $M_{p,q}$, which is proved in Lemma 3.2 of \cite{raynaud2002mazurmap} (for an independant proof in the case $L_{p}(\mathcal{M},\tau)=S_{p}$, see \cite{puschnigg2008finitely}).

\begin{prp}{\rm \cite{raynaud2002mazurmap}} 
For $1\leq p,q<\infty$, the Mazur map $M_{p,q}$ is uniformly continuous on the unit sphere $S(L_{p}(\mathcal{M}))$.
\end{prp}

\section{\sc{Group representations on $L_{p}(\mathcal{M})$}}
Sherman's description of the surjective isometries of $L_{p}(\mathcal{M})$ in \cite{sherman2005isometries} is a crucial tool in the following result (non surjective isometries in the semi-finite case, and 2-isometries in the general case are described in \cite{yeadon1980isometries} and \cite{Junge2005isometries} respectively). This will allow us to transfer a representation of a group $G$ on $L_{p}(\mathcal{M})$ to a representation of $G$ on $L_{2}(\mathcal{M})$.
\begin{prp}\label{prp7}
For $p>2$, and $U\in O(L_{p}(\mathcal{M}))$, the map $V=M_{p,2}\circ U\circ M_{2,p}$ belongs to $O(L_{2}(\mathcal{M}))$.
\end{prp}

\begin{proof}
The fact that $\n V(x)\n_{2}=\n x\n_{2}$ for all $x\in L_{2}(\mathcal{M})$ follows from Proposition \ref{prp3.5}, and $V$ is bijective by Lemma \ref{lem3.4}. We have to prove that $V$ is linear on $L_{2}(\mathcal{M})$. \\
By Theorem 1.2 in \cite{sherman2005isometries}, there exist a Jordan-isomorphism  $J$ of $\mathcal{M}$ and a unitary $w\in\mathcal{M}$ such that
$$U(\varphi^{1/p})=w(\varphi\circ J^{-1})^{1/p}\textrm{ for all }\varphi\in\mathcal{M}_{*}^{+}.$$ 
It was shown in \cite{watanabe1996prolongementJ} that $J$ extends to a Jordan-{*}-isomorphism $\widetilde{J}$ between $L_{0}(\mathcal{N}_{\varphi_{0}},\tau_{\varphi_{0}})$ and $L_{0}(\mathcal{N}_{\varphi_{0}\circ J^{-1}},\tau_{\varphi_{0}\circ J^{-1}})$; moreover, $\widetilde{J}$ is an extension of an isomorphism between $\mathcal{N}_{\varphi_{0}}$ and $\mathcal{N}_{\varphi_{0}\circ J^{-1}}$ as well as a homeomorphism for the measure topology on $L_{0}(\mathcal{N}_{\varphi_{0}},\tau_{\varphi_{0}})$ and $L_{0}(\mathcal{N}_{\varphi_{0}\circ J^{-1}},\tau_{\varphi_{0}\circ J^{-1}})$. The isomorphism $\widetilde{J}$ satisfies the relations
\begin{displaymath}
\begin{split}
&\tau_{\varphi_{0}}\circ \widetilde{J}^{-1}=\tau_{\varphi_{0}\circ J^{-1}}\\
&J^{-1}\circ\Phi_{\varphi_{0}\circ J^{-1}}=\Phi_{\varphi_{0}}\circ\widetilde{J}^{-1}
\end{split}
\end{displaymath}
\begin{lem}\label{J}
For $\varphi\in\mathcal{M}_{*}^{+}$, we have $$\frac{d\tilde{\varphi}^{\varphi_{0}}}{d\tau_{\varphi_{0}}}=\widetilde{J}^{-1}(\frac{d\tilde{\varphi\circ J^{-1}}^{\varphi_{0}\circ J^{-1}}}{d\tau_{\varphi_{0}\circ J^{-1}}}).$$
\end{lem}
\begin{proof}
For all $\varphi\in\mathcal{M}_{*}^{+}$, we have
\begin{displaymath}
\begin{split}
\tau_{\varphi_{0}}(\frac{d\tilde{\varphi}^{\varphi_{0}}}{d\tau_{\varphi_{0}}}\esp .\esp )&=\varphi\circ\Phi_{\varphi_{0}}\\
&=\varphi\circ J^{-1}\circ\Phi_{\varphi_{0}\circ J^{-1}}\circ\widetilde{J}\\
&=\tau_{\varphi_{0}\circ J^{-1}}(\frac{d\tilde{\varphi\circ J^{-1}}^{\varphi_{0}\circ J^{-1}}}{d\tau_{\varphi_{0}\circ J^{-1}}} \widetilde{J}(\esp.\esp))\\
&=\tau_{\varphi_{0}}\circ \widetilde{J}^{-1}(\frac{d\tilde{\varphi\circ J^{-1}}^{\varphi_{0}\circ J^{-1}}}{d\tau_{\varphi_{0}\circ J^{-1}}} \widetilde{J}(\esp.\esp))\\
&=\tau_{\varphi_{0}}(\widetilde{J}^{-1}(\frac{d\tilde{\varphi\circ J^{-1}}^{\varphi_{0}\circ J^{-1}}}{d\tau_{\varphi_{0}\circ J^{-1}}})\esp.\esp)\esp,
\end{split}
\end{displaymath}
where in the last equality we used the fact that $\widetilde{J}$ is Jordan homomorphism.
\end{proof}
In Lemma 2.1 in \cite{watanabe1992poids...}, it is shown that there exists a topological $*$-isomorphism $\widetilde{\mathcal{K}}$ between $L_{0}(\mathcal{N}_{\varphi_{0}},\tau_{\varphi_{0}})$ and $L_{0}(\mathcal{N}_{\varphi_{0}\circ J^{-1}},\tau_{\varphi_{0}\circ J^{-1}})$ which satisfies the following relation on the Radon-Nikodym derivatives :
$$ \widetilde{\mathcal{K}}(\frac{d\tilde{\varphi}^{\varphi_{0}}}{d\tau_{\varphi_{0}}})=\frac{d\tilde{\varphi}^{\varphi_{0}\circ J^{-1}}}{d\tau_{\varphi_{0}\circ J^{-1}}}\textrm{ for all }\varphi\in\mathcal{M}_{*}^{+}.$$
From Lemma \ref{J}, we obtain
$$\frac{d\tilde{\varphi\circ J^{-1}}^{\varphi_{0}}}{d\tau_{\varphi_{0}}}=\widetilde{\mathcal{K}}^{-1}\circ \widetilde{J}(\frac{d\tilde{\varphi}^{\varphi_{0}}}{d\tau_{\varphi_{0}}})\textrm{ for all }\varphi\in\mathcal{M}_{*}^{+}.$$
As a consequence, the linear and bijective isometry $U$ of $L_{p}(\mathcal{M})$ is given by the following relation on positive elements :
$$U(x)=w\esp(\widetilde{\mathcal{K}}^{-1}\circ \widetilde{J}(x))\textrm{ for all }x\in L_{p}(\mathcal{M})^{+}. $$
This relation extends by linearity to the whole $L_{p}(\mathcal{M})$.\\

Now notice that $\widetilde{\mathcal{K}}^{-1}\circ \widetilde{J}$ is a Jordan-isomorphism on $\mathcal{N}_{\varphi_{0}}$ and a topological isomorphism (for the measure topology) on $L_{0}(\mathcal{N}_{\varphi_{0}},\tau_{\varphi_{0}})$. By Proposition \ref{maz}, for $x\in\mathcal{N}_{\varphi_{0}}$, we have
\begin{displaymath}
\begin{split}
V(x)&=M_{p,2}\circ U\circ M_{2,p}(x)\\
&=w(M_{p,2}\circ\widetilde{\mathcal{K}}^{-1}\circ \widetilde{J}\circ M_{2,p}(x))\\
&=w(\widetilde{\mathcal{K}}^{-1}\circ \widetilde{J}(x)).
\end{split}
\end{displaymath}
Recall from \cite{raynaud2002mazurmap} that the Mazur map is continuous for the measure topology on $L_{0}(\mathcal{N}_{\varphi_{0}},\tau_{\varphi_{0}})$. So by density of $\mathcal{N}_{\varphi_{0}}$ in $L_{0}(\mathcal{N}_{\varphi_{0}},\tau_{\varphi_{0}})$ for the measure topology, we have
$$V(x)=w(\widetilde{\mathcal{K}}^{-1}\circ \widetilde{J}(x))\textrm{ for all }x\in L_{2}(\mathcal{M})$$
which gives the linearity of $V$ on $L_{2}(\mathcal{M})$. 
\end{proof}

\begin{rem}
{\rm The proof of the linearity of the map }$V$ {\rm in Proposition \ref{prp7} is simpler in the case where} $\mathcal{M}$ {\rm is a von Neumann algebra equipped with a faithful semi-finite normal trace} $\tau$. {\rm Indeed, by Theorem 2 in \cite{yeadon1980isometries}, there exist a Jordan-isomorphism }$J${\rm, a positive operator }$B$ {\rm commuting with }$J(\mathcal{M})${\rm, and a partial isometry }$W$ {\rm in }$\mathcal{M}$ {\rm with the property that }$W^{*}W$ {\rm is the support of }$B${\rm, such that} 
$$U(x)=WBJ(x)\esp{\rm for\esp all }\esp x\in\mathcal{M}\cap L_{p}(\mathcal{M},\tau).$$
{\rm Using the fact that }$B$ {\rm commutes with }$J(\mathcal{M})${\rm, and as in the proof of Proposition \ref{maz}, for all }$x=\alpha\vert x\vert\in \mathcal{M}\cap L_{p}(\mathcal{M},\tau)${\rm, we have}
\begin{displaymath}
\begin{split}
V(x)&=WM_{p,2}(BJ_{1}(\alpha\vert x\vert^{\frac{p}{2}})+BJ_{2}(\alpha\vert x\vert^{\frac{p}{2}}))\\
&=WM_{p,2}(BJ_{1}(\alpha\vert x\vert^{\frac{p}{2}}))+WM_{p,2}(BJ_{2}(\alpha\vert x\vert^{\frac{p}{2}}))\\
&=WJ_{1}(\alpha)B^{\frac{p}{2}}J_{1}(\vert x\vert)+WJ_{2}(\alpha)B^{\frac{p}{2}}J_{2}(\alpha\vert x\vert\alpha^{*})\\
&=WB^{\frac{p}{2}}J(x).
\end{split}
\end{displaymath}
{\rm The linearity on the whole }$L_{p}(\mathcal{M},\tau)$ {\rm follows from the density of }$\mathcal{M}\cap L_{p}(\mathcal{M},\tau)$ {\rm in }$L_{p}(\mathcal{M},\tau)$.
\end{rem}

\begin{cor}\label{cor}
Let $G$ be a topological group, $p\geq2$, and $U:G\rightarrow O(L_{p}(\mathcal{M}))$ be a representation on $L_{p}(\mathcal{M})$. For $g\in G$, define $V(g):L_{2}(\mathcal{M})\rightarrow L_{2}(\mathcal{M})$ by
$$V(g)=M_{p,2}\circ U(g)\circ M_{2,p}.$$
Then $V$ is a representation of $G$ on $L_{2}(\mathcal{M})$.
\end{cor}
\begin{proof}
By the previous proposition, $V(g)\in O(L_{2}(\mathcal{M}))$ for every $g$ in $G$. Moreover, the map $g\mapsto V(g)x$ is continuous, since $g\mapsto U(g)M_{2,p}(x)$ is continuous and since $M_{p,2}:L_{p}(\mathcal{M})\rightarrow L_{2}(\mathcal{M})$ is continuous.\\
It remains to check that $V$ is a homomorphism. For this, let $g_{1},g_{2}\in G$. Then, by Lemma \ref{lem3.4},
\begin{displaymath}
\begin{split}
V(g_{1})V(g_{2})&=M_{p,2}\circ U(g_{1})\circ M_{2,p}\circ M_{p,2}\circ U(g_{2})\circ M_{2,p}\\
&=M_{p,2}\circ U(g_{1})\circ U(g_{2})\circ M_{2,p}\\
&=M_{p,2}\circ U(g_{1}g_{2})\circ M_{2,p}\\
&=V(g_{1}g_{2}).
\end{split}
\end{displaymath}
\end{proof}

Let $U$ be a representation of a topological group $G$ on $L_{p}(\mathcal{M})$ and let
$$L_{p}(\mathcal{M})^{U(G)}=\{x\in L_{p}(\mathcal{M})\esp\vert\esp U(g)x=x\textrm{ for all }g\in G\esp\}$$
be the space of $U(G)$-invariant vectors in $L_{p}(\mathcal{M})$. Let $p'$ be the conjugate of $p$ and $U^{*}$ the contragredient representation of $U$ on the dual space $L_{p'}(\mathcal{M})$ of $L_{p}(\mathcal{M})$. Since $L_{p}(\mathcal{M})$ is superreflexive, there exists a complement $L_{p}(\mathcal{M})^{'}$ for $L_{p}(\mathcal{M})^{U(G)}$ (see Proposition 2.6 in \cite{bader2007propertyTLp}) and we have 
$$L_{p}(\mathcal{M})^{'}=\{v\in L_{p}(\mathcal{M})\esp\vert\esp{\rm Tr}(vc)=0\esp\textrm{for all}\esp c\in L_{p'}(\mathcal{M})^{U^{*}(G)}\}.$$
\begin{prp}\label{prp}\label{prp8}
Let $v\in S(L_{p}(\mathcal{M})^{'})$, then 
$$d(v,L_{p}(\mathcal{M})^{U(G)})\geq\frac{1}{2}.$$
\end{prp}
\begin{proof}
Assume, by contradiction, that there exists $b\in L_{p}(\mathcal{M})^{U(G)}$ such that
$$\vert\vert v-b\vert\vert_{p}<\frac{1}{2}.$$
Then $\frac{1}{2}\leq\vert\vert b\vert\vert_{p}\leq\frac{3}{2}$. Setting $c=\dfrac{b}{\vert\vert b\vert\vert_{p}}$, we have $\vert\vert b-c\vert\vert_{p}\leq\frac{1}{2}$.\\
Since $c\in L_{p}(\mathcal{M})^{U(G)}$, it is easily checked that $M_{p,p'}(c)^{*}\in L_{p'}(\mathcal{M})^{U^{*}(G)}$; hence
$${\rm Tr}((c-v)M_{p,p'}(c)^{*})={\rm Tr}(cM_{p,p'}(c)^{*})=\vert\vert c\vert\vert_{p}^{p}=1.$$
On the other hand, using H\"{o}lder's inequality, we have
\begin{displaymath}
\begin{split}
1&={\rm Tr}((c-v)M_{p,p'}(c)^{*})\\
&\leq\vert\vert c-v\vert\vert_{p}\vert\vert M_{p,p'}(c)^{*}\vert\vert_{p'}\\
&=\vert\vert c-v\vert\vert_{p}\vert\vert c\vert\vert_{p}^{\frac{p}{p'}}\\
&=\vert\vert c-v\vert\vert_{p}.
\end{split}
\end{displaymath}
This implies that
\begin{displaymath}
\begin{split}
\vert\vert v-b\vert\vert_{p}&\geq\vert\vert v-c\vert\vert_{p}-\vert\vert c-b\vert\vert_{p}\\
&\geq\frac{1}{2}
\end{split}
\end{displaymath}
and this is a contradiction.
\end{proof}

\section{\sc{Proof of Theorem \ref{thm1}}}

We follow the strategy of the proof of Theorem A in \cite{bader2007propertyTLp}. Let $p\in]1,\infty[$. and let $U$ be a representation on $L_{p}(\mathcal{M})$ of a group $G$. Let $H$ be a closed subgroup of $G$ such that the pair $(G,H)$ has property $(T)$. We claim that the representation $U'$ of $G$ on the complement $L_{p}(\mathcal{M})^{'}$ of $L_{p}(\mathcal{M})^{U(H)}$ has no almost $U'(G)$-invariant vectors. This will prove Theorem \ref{thm1}.\\
Let $Q$ be a compact subset in $G$, and take $\epsilon>0$. Assume by contradiction that there exists almost $U(G)$-invariant vectors in $L_{p}(\mathcal{M})^{'}$. Then, we can find, for every $n$, a unit vector $v_{n}$ such that 
$$\sup_{g\in Q}\vert\vert U(g)v_{n}-v_{n}\vert\vert_{p}<\frac{1}{n}.$$\\
By Corollary \ref{cor}, $V=M_{p,2}\circ U\circ M_{2,p}$ defines a representation of $G$ on $L_{2}(\mathcal{M})$.
Let $w_{n}$ be the orthogonal projection of $M_{p,2}(v_{n})$ on the orthogonal complement $L_{2}(\mathcal{M})^{'}$ of $L_{2}(\mathcal{M})^{V(H)}$. We claim that $w_{n}$ is $(Q,\epsilon)$-invariant for $V$ for $n$ sufficiently large. This will contradict property $(T)$ for the pair $(G,H)$.\\

We first show that there exists $\delta^{'}>0$ such that 
$$d(M_{p,2}(v_{n}),L_{2}(\mathcal{M})^{V(H)})\geq\delta^{'}\esp\textrm{for all}\esp n.$$
Indeed, otherwise for some $n$, there exists $a_{k}\in L_{2}(\mathcal{M})^{V(H)}$ such that $$\vert\vert M_{p,2}(v_{n})-a_{k}\vert\vert_{2}\xrightarrow[k \rightarrow\infty]{} 0.$$
By Proposition \ref{prp3.5}, we have
$$\vert\vert M_{p,2}(v_{n})\vert\vert_{2}=\vert\vert v_{n}\vert\vert_{p}^{\frac{p}{2}}=1. $$
Since $\vert\vert a_{k}\vert\vert_{2}\xrightarrow[k \rightarrow\infty]{}\vert\vert M_{p,2}(v_{n})\vert\vert_{2}=1$, we can assume that $\vert\vert a_{k}\vert\vert_{2}=1$. Notice that 
$$M_{2,p}(L_{2}(\mathcal{M})^{V(H)})=L_{p}(\mathcal{M})^{U(H)}.$$
Hence, $M_{2,p}(a_{k})$ belongs to $L_{p}(\mathcal{M})^{U(H)}$ for every $k$. Moreover $$\vert\vert v_{n}-M_{2,p}(a_{k})\vert\vert_{p}\xrightarrow[k\rightarrow\infty]{}0$$ by the uniform continuity of $M_{2,p}$ on the unit sphere (see Proposition \ref{prp3}). This is a contradiction to Proposition \ref{prp8}. \\
In particular, we have 
$$\vert\vert w_{n}\vert\vert_{2}=d(M_{p,2}(v_{n}),L_{2}(\mathcal{M})^{V(H)})\geq\delta^{'}.$$

For $g\in Q$, we have
\begin{displaymath}
\begin{split}
\vert\vert V(g)w_{n}-w_{n}\vert\vert_{2}&\leq\vert\vert V(g)M_{p,2}(v_{n})-M_{p,2}(v_{n})\vert\vert_{2}\\
&=\vert\vert M_{p,2}(U(g)v_{n})-M_{p,2}(v_{n})\vert\vert_{2}.
\end{split}
\end{displaymath}
Recall that $\vert\vert v_{n}\vert\vert_{p}^{\frac{p}{2}}=1$ and that
$$\sup_{g\in Q}\vert\vert U(g)v_{n}-v_{n}\vert\vert_{p}<\frac{1}{n}.$$
Hence, by the uniform continuity of $M_{p,2}$ on $S(L_{2}(\mathcal{M}))$, there exists an integer $N$ (depending only on ($Q,\epsilon$)) such that
$$\sup_{g\in Q}\vert\vert V(g)w_{n}-w_{n}\vert\vert_{2}<\epsilon \delta^{'}\esp\textrm{for}\esp n\geq N.$$
Since $\vert\vert w_{n}\vert\vert_{2}\geq\delta^{'}$, it follows that
$$\sup_{g\in Q}\vert\vert V(g)w_{n}-w_{n}\vert\vert_{2}<\epsilon\vert\vert w_{n}\vert\vert_{2}\esp\textrm{for}\esp n\geq N.$$
This shows that $w_{n}$ is $(Q,\epsilon)$-invariant for $U$ when $n\geq N$. This finishes the proof of Theorem \ref{thm1}.

\section{\sc{Property $(F_{L_{p}(\mathcal{M})})$ for higher rank groups}}
Let $H$ be a closed normal subgroup of $G$ and let $L$ be a closed group of $G$. Assume that $G=L\ltimes H$. The following strong relative property $(T_{B})$ was considered in \cite{bader2007propertyTLp} :
\begin{df}
\rm{}A pair $(L\ltimes H,H)$ has property $(T_{B})$ if, for any orthogonal representation $\rho:L\ltimes H\rightarrow O(B)$, the quotient representation $\rho{'}:L\rightarrow O(B/B^{\rho(H)})$ does not almost have $\rho{'}(L)$-invariant vectors.
\end{df}
A straightforward modification of our proof of Theorem \ref{thm1} shows that we also have the following result :
\begin{thm}\label{thm1.5}
Let $(L\ltimes H,H)$ be a pair with strong relative property $(T)$. Then $(L\ltimes H,H)$ has strong relative property $(T_{L_{p}(\mathcal{M})})$ for $1<p<\infty$.
\end{thm}
Let $G$ be a higher rank group as defined in the introduction. Using an analogue of Howe-Moore's theorem on vanishing of matrix coefficients, the authors of  \cite{bader2007propertyTLp} showed that $G$ has property $(F_{B})$ whenever $B$ is a superreflexive Banach space and a certain pair $(L\ltimes H,H)$ of subgroups, which has property $(T)$, has also $(T_{B})$. The property $(F_{L_{p}(\mathcal{M})})$ for higher rank groups in Theorem 1.6 is then a consequence of Theorem 5.2. Moreover, the result for lattices in higher rank groups is obtained by an induction process exactly as in the Proposition 8.8 of \cite{bader2007propertyTLp}.

\section{\sc{Acknowledgements}}
We wish to thank Bachir Bekka for all his very useful advice and the IRMAR for the stimulating atmosphere and the quality of working conditions. We are also grateful to Masato Mimura for very useful discussions.

\bibliographystyle{plain}
\bibliography{biblio}

\end{document}